\documentclass[11pt]{amsart}

\usepackage[utf8]{inputenc}
\usepackage[english]{babel}

\usepackage{url}
\usepackage{anysize}
\usepackage{amsfonts}
\usepackage{amssymb}
\usepackage{amsmath}
\usepackage{amsthm}
\usepackage{tikz-cd}
\usepackage[alphabetic,backrefs,lite]{amsrefs} % for bibliography
\usepackage[all]{xy}
\usepackage{enumerate}
\usepackage{multirow}
\usepackage{ulem}
\usepackage{geometry}
\usepackage{color}
\usepackage{colonequals}
\usepackage{mathrsfs}

\definecolor{darkgreen}{rgb}{0,0.5,0}

\usepackage[
draft=false,
colorlinks, citecolor=darkgreen,      
pdfauthor={S. Brivio, F. Fallucca, and F. F. Favale},
pdftitle={On some special subvarieties of moduli spaces of vector bundles},
backref,
linktocpage]{hyperref}

\usepackage{tcolorbox}
\newtcolorbox{mybox}{colback=blue!5!white,colframe=blue!75!black}

\newtheorem{theorem}{Theorem}[section]
\newtheorem*{theorem*}{Theorem}
\newtheorem{proposition}[theorem]{Proposition}
\newtheorem{definition}[theorem]{Definition}

\newtheorem{lemma}[theorem]{Lemma}

\newtheorem{remark}[theorem]{Remark}

\DeclareMathOperator{\codim}{codim}

\DeclareMathOperator{\Ker}{Ker}

\DeclareMathOperator{\Ext}{Ext}
\DeclareMathOperator{\Hom}{Hom}
\DeclareMathOperator{\Pic}{Pic}

\DeclareMathOperator{\rk}{rk}

\DeclareMathOperator{\End}{End}

\DeclareMathOperator{\id}{id}
\newcommand{\ev}{ev}

\DeclareMathOperator{\Gr}{Gr}
\DeclareMathOperator{\Ima}{Im}

\DeclareMathOperator{\gr}{gr}
\DeclareMathOperator{\edim}{edim}

\newcommand{\nc}{\newcommand}

\nc{\cH}{{\mathcal H}}
\nc{\cA}{{\mathcal A}}
\nc{\cG}{{\mathcal G}}
\nc{\cC}{{\mathcal C}}
\nc{\cD}{{\mathcal D}}
\nc{\cO}{{\mathcal O}}
\nc{\cI}{{\mathcal I}}
\nc{\cB}{{\mathcal B}}
\nc{\cY}{{\mathcal Y}}
\nc{\cK}{{\mathcal K}} 
\nc{\cX}{{\mathcal X}}
\nc{\cS}{{\mathcal S}}
\nc{\cE}{{\mathcal E}}
\nc{\cF}{{\mathcal F}}
\nc{\cZ}{{\mathcal Z}}
\nc{\cQ}{{\mathcal Q}}
\nc{\cN}{{\mathcal N}}
\nc{\cP}{{\mathcal P}}
\nc{\cL}{{\mathcal L}}
\nc{\cM}{{\mathcal M}}
\nc{\cT}{{\mathcal T}}
\nc{\cW}{{\mathcal W}}
\nc{\cU}{{\mathcal U}}
\nc{\cJ}{{\mathcal J}}
\nc{\cV}{{\mathcal V}}
\nc{\cR}{{\mathcal R}}
\nc{\bH}{{\mathbb H}}
\nc{\bA}{{\mathbb A}}
\nc{\bG}{{\mathbb G}}
\nc{\bC}{{\mathbb C}}
\nc{\bO}{{\mathbb O}}
\nc{\bI}{{\mathbb I}}
\nc{\bB}{{\mathbb B}}
\nc{\bY}{{\mathbb Y}}
\nc{\bK}{{\mathbb K}} 
\nc{\bX}{{\mathbb X}}
\nc{\bS}{{\mathbb S}}
\nc{\bE}{{\mathbb E}}
\nc{\bF}{{\mathbb F}}
\nc{\bZ}{{\mathbb Z}}
\nc{\bQ}{{\mathbb Q}}
\nc{\bN}{{\mathbb N}}
\nc{\bP}{{\mathbb P}}
\nc{\bL}{{\mathbb L}}
\nc{\bM}{{\mathbb M}}
\nc{\bT}{{\mathbb T}}
\nc{\bW}{{\mathbb W}}
\nc{\bU}{{\mathbb U}}
\nc{\bD}{{\mathbb D}}
\nc{\bJ}{{\mathbb J}}
\nc{\bV}{{\mathbb V}}
\nc{\bR}{{\mathbb R}}

\newcommand{\cHom}{\mathcal{H}om}

\begin{document}
	
	\title[Some rational  subvarieties of moduli spaces of stable vector bundles]
	{Some rational subvarieties of moduli spaces of stable vector bundles}
	
	\author{Sonia Brivio}
	\address{Dipartimento di Matematica e Applicazioni,
		Universit\`a degli Studi di Milano-Bicocca,
		Via Roberto Cozzi, 55,
		I-20125 Milano, Italy}
	\email{sonia.brivio@unimib.it}
	
	\author{Federico Fallucca}
	\address{Dipartimento di Matematica,
		Universit\`a degli Studi di Trento,
		via Sommarive, 14,
		I-38123 Trento, Italy}
	\email{federico.fallucca@unitn.it}
	\email{fallucca@altamatematica.it}
	
	\author{Filippo F. Favale}
	\address{Dipartimento di Matematica,
		Universit\`a degli Studi di Pavia,
		via Ferrata, 5
		I-27100 Pavia, Italy}
	\email{filippo.favale@unipv.it}
	
	%\date{\today}
	\thanks{
        %\noindent 
        {\bf Acknowledgements}:
        The authors are partially supported by INdAM-GNSAGA. The second author held a research grant from INdAM, Istituto Nazionale di Alta Matematica.
}

\begin{abstract}
Let $X$ be a smooth complex  irreducible projective variety  of dimension  $n \geq 2$ and $H$  be an ample line bundle on $X$. In this paper, we construct families of $\mu_H$-stable vector bundles on $X$ having fixed determinant and rank $r$, which are generated by $r+1$ global sections, parametrized by Grassmanian varieties.  This gives into the corresponding moduli spaces  special subvarieties birational to Grassmannian. 
\end{abstract}
	
\maketitle

%-------------------------------------------------------
%-------------------------------------------------------
%-------------------------------------------------------
	
\section*{Introduction}

The notion of $\mu$-stability for vector bundles on curves was introduced by Mumford, and subsequently extended to higher-dimensional varieties by the foundational works of Takemoto, Gieseker and Maruyama. In particular, Maruyama proved the existence of coarse moduli spaces parametrising isomorphism classes of $\mu_H$-stable vector bundles with respect to an ample polarisation $H$, on a smooth projective variety (see \cite{Mar77}). 
\smallskip

While the case of curves is nowadays well understood, the situation in higher dimension remains considerably less developed. In particular, there are no general results ensuring the non-emptyness of these moduli spaces. For this reason, explicit constructions of families of $\mu$-stable vector bundles dominating particular subvarieties of these moduli spaces seem to be of significant interest.
\smallskip

Let $X$ be a smooth complex irreducible projective variety of dimension $n \geq 2$ and let $L$ be a non-trivial globally generated line bundle on $X$. In this paper, our aim is to produce families of vector bundles on $X$ with rank $r \geq 2$ and determinant $L$, which are generated by $r+1$ global sections and are $\mu_H$-stable with respect to an ample line bundle $H$ on $X$. Moreover, these families give rise to subvarieties in the corresponding moduli spaces which are birational to a Grassmannian variety.
\smallskip

Our construction starts as follows. Let $W \subset H^0(L)$ be a $(r+1)$-dimensional subspace such that the evaluation map of global sections $W \otimes {\mathcal O}_X \to L$ is a surjective map of vector bundles on $X$. Denote by $M_{W,L}$ its kernel; it is then a vector bundle on $X$ of rank $r$ and determinant $L^{-1}$.
Its dual is a vector bundle $E_W$ too, with rank $r$, determinant $L$, and Chern classes $\underline{c}=(c_1(L),\dots, c_1(L)^n)$ (see Lemma \ref{lemmatecnico2}), which fit into the following exact sequence:
$$ 0 \to L^{-1} \to W \otimes {\mathcal O}_X \to E_W \to 0.$$    
If  $M_{W,L}$ is $\mu_H$ semistable for an ample line bundle $H$ on $X$, then so is $E_W$  and it is generated by $r+1$ global sections.

Vector bundles of the form $M_{W,F}$ (denoted as $M_F$ in the complete case  $W= H^0(F)$), arising as kernels of evaluation map of globally generated vector bundles $F$, on a smooth variety, are known in literature as  {\it kernel bundles, dual span bundles and sygyzy bundles}. Their stability has been extensively studied. 
For  a smooth curve of genus $g \geq 2$, the theory is well developed at least for the complete case (see, for example, the results in \cite {But94}, \cite{Mis08}, \cite{EL92}, \cite{CH25}, \cite{BBN08}); there are also  some results in the case of singular curves (see for example \cite{BF20}). In higher dimension,  only partial results are available, mainly in  the complete case and for line bundles (see \cite{Fle84} 
and    \cite{ELM13} and \cite{Cam12}).

Our strategy for  proving the  stability of $M_{W,L}$ consists in reducing the problem  to the stability of kernel bundles on smooth curves. More precisely, let $H$ be an ample line bundle on $X$ and assume that there exists a smooth curve $C \subset X$ of genus $g \geq 2$, given as   a complete intersection of divisors of $\vert H \vert$,  such that the restriction map of global section $H^0(X,L) \to H^0(C, L_{\vert C})$ is surjective.  We can prove that the restriction of $M_{W,L}$ to $C$ is a kernel bundle on $C$ and its   stability implies  $\mu_H$-stability of $M_{W,L}$.
Stability  on the curve $C$ is ensured  by requiring suitable numerical  assumptions on the degree of $L_{\vert C}$. Specifically, our result holds whenever either our conditions or those established in \cite{Mis08} are satisfied.

We will say that the data $(X,L,H,r)$  is {\it admissible} if the above mentioned assumptions are satisfied (c.f. Definition \ref{DEF:AdmissibleTriple}).  
We denote by ${\mathcal M}_H^s(r,L, \underline{c})$ the moduli space parametrizing  $\mu_H$-stable vector bundles with rank $r$, determinant $L$, and Chern classes $\underline{c}$ depending on $L$ (c.f. Definition \ref{DEF:definizione di U}). 
Our main result is the following (see Theorem \ref{thm:inj_mor_moduli}):
 \begin{theorem*} 
 Let $(X,L,H,r)$ an admissible collection, then the moduli space ${\mathcal M}_H^s(r,L, \underline{c})$ is non-empty and it contains a subvariety birational to the Grassmannian variety $\Gr(r+1,H^0(L))$.
\end{theorem*}
This provides, in arbitrary dimension, a systematic method to construct globally generated $\mu_H$-stable vector bundles with prescribed determinant and Chern classes. 
\smallskip

In the second part of the paper, we specialise to algebraic surfaces, and we investigate the scope of our construction through a series of examples. We exhibit admissible collections with surfaces for each  Kodaira  dimensions $\kappa(S)\in \{-\infty, 0,1,2\}$.
Of particular interest is the case of K3 surfaces. Indeed,  when $S$ is a K3 surface and $H$ is an ample primitive line bundle on $S$, the subvariety arising from our construction turns out to be a Lagrangian subvariety of the moduli space, provided the latter is a smooth irreducible symplectic variety (see Theorem \ref{THm:lagrangian} and Remark \ref{REM:Lagrangian}).

%-------------------------------------------------------
%-------------------------------------------------------
	
	\section{Notations and preliminary results}
	\label{SEC:preliminary results}
	\subsection{Moduli spaces of stable sheaves}
	Let  $X$ be a smooth irreducible projective complex variety of dimension $n \geq 2$ and $H$ an ample line bundle on $X$. We will need to deal with moduli spaces parametrising ($H$-stable) vector bundles on $X$. In this section, we recall some well-known results on this topic. Our main reference is \cite{HL}. To begin with, we recall that - unlike in the case of curves - obtaining a projective moduli space requires us to include torsion-free sheaves on $X$.
    \smallskip

    Let $E$  be a non-trivial torsion-free sheaf on $X$. There exist a non empty open subset $ U \subseteq X$ such that $E|_{U}$ is  a vector bundle. Then $\rk(E)$ is defined as the rank of $E|_{U}$.  When the pair $(X,H)$ is fixed, one can define the $\mu_H$-semistability and $H$-semistability through the $H$-slope $\mu_H$ of $E$ and its reduced Hilbert polynomial, respectively. We recall that the $\mu_H$-slope is $$\mu_H(E)= \frac{c_1(E) \cdot H^{n-1}}{\rk(E)}$$ 
	whereas the reduced Hilbert polynomial is, up to a positive constant which depends only on the pair $(X,H)$, 
	$$p_H(E,k)= \frac{\chi(E \otimes H^{\otimes k })}{\rk(E)}.$$
	A torsion-free sheaf $E$ is called $\mu_H$-{\it semistable} if for any non zero subsheaf $F \subset E$  with $\rk(F) < \rk(E)$ we have $\mu_H(F) \leq \mu_H(E)$, it is said $\mu_H$-stable it the strict inequality holds.

    The sheaf $E$ is said $H$-semistable if for any non-zero subsheaf $F \subset E$ we have $p_H(F,k) \leq p_H(E,k)$ for $k>>0$ and it is said $H$-stable if the strict inequality holds for any proper subsheaf $F$. One has the following chain of implications:

    \begin{center}
    $E$ is $\mu_H$-stable  $\implies$  $E$ is $H$-stable $\implies$  $E$ is  $H$-semistable  $\implies$  $E$ is  $\mu_H$-semistable.
	\end{center}

    Any line bundle on $X$ is $\mu_H$-stable. Taking duals and tensoring by line bundles preserve both $H$-semistability and $H$-stability. Moreover, the sum of two $\mu_H$-semistable vector bundles is $\mu_H$-semistable if and only if they have the same $H$-slope. 
    \smallskip

    Any $H$-semistable torsion free sheaf $E$ admits a Jordan-Holder fibration
	$$ (JH) \quad 0 = E_0 \subset E_1 \subset  \cdots  \subset E_l = E$$
	with $\gr(E_i)=\frac{E_i}{E_{i-1}}$ which is  $H$-stable with reduced Hilbert polynomial $p_H(E,k)$.  So one can define the graded object   $\gr(E)= \bigoplus \gr(E_i)$. Two $H$-semistable torsion-free sheaves are said $S$-equivalent if they have isomorphic graded objects.
    \smallskip

    Let $P(k) \in \mathbb Q[k]$  be  a polynomial of degree $n$, and denote by 
	$\cM_H(P)$ the moduli space parametrizing $S$-equivalence classes of $H$-semistable torsion free sheaves $E$ on $X$
	with Hilbert polinomial (with respect to the polarization $H$) given by $P_H(E)=P$. The existence of this moduli space is guaranteed, for example, by \cite[Theorem 3.4.4]{HL}.
	$\cM_H(P)$ is a projective scheme, containing as an open subscheme 
	the moduli space $\cM^s_H(P)$ parametrizing  isomorphism 
	classes of  $\mu_H$-stable  vector bundles. 
	Finally, if $\underline{c}=(c_1,c_2, \cdots, c_n)$ with $c_i\in H^{2i}(X,\bZ)$, $\cM^s_H(P)$ is a disjoint union of schemes 
	$\cM^s_H(r,\underline{c})$, where $\cM^s_H(r,\underline{c})$ is the moduli space  of $\mu_H$-stable vector bundles on $X$  of rank $r$ with Chern classes $(c_1, c_2, \cdots , c_{n})$ up to numerical equivalence (see \cite{Mar77}).
	%Maruyama, Moduli of stable sheaves; I, J. Math. Kyoto Univ., 17 (1977), 91-126.
	We recall that by Bogomolov's inequality if $E$ is a  torsion free  $\mu_H$-semistable sheaf  of rank $r$ on $X$ we have 
	\begin{equation}
    \label{EQ:Discriminant}
	\Delta_H(E) = (2rc_2(E)- (r-1)c_1(E)^2)\cdot H^{n-2} \geq 0;
	\end{equation}
	this was proved by Bogomolov \cite{Bog78} for surfaces and generalized to higher dimensional smooth projective varieties using Mumford-Mehta-Ramanathan restriction theorem \cite{MR84}. 
    
    \smallskip
	Let $L \in \Pic(X)$ be a line bundle. We denote by $\mathcal M^s_H(r,L,\underline{c})$ the moduli space of  $\mu_H$-stable vector bundles  $E$ with  $\det E = L$ and  Chern classes $c_i(E)= c_i$, $i = 2, \cdots n$. This is simply the fiber at $L$ of the morphism $\det \colon \cM^s_H(r, \underline{c}) \to \Pic (X)$ which sends $[E]$ to its determinant $\det(E)$. 
	Finally, we recall the following properties concerning the infinitesimal structure of these moduli spaces.
	Assume that  there exists $[E] \in \cM^s_H(r,L,\underline{c})$, which is   the isomorphism class of a $\mu_H$-stable vector bundle. Then
	$$T_{[E]}(\cM^s_H(r,L,\underline{c})) \simeq \Ext^1(E,E)_0,$$ 
	$$ \dim \Ext^1(E,E)_0 - \dim \Ext^2(E,E)_0 \leq 
	\dim_{[E]} \cM^s_H(r,L,\underline{c}) \leq \dim \Ext^1(E,E)_0,$$
	where 
	$\Ext^i(E,E)_0$ is the  kernel of the map  $h^i(tr) \colon \Ext^i(E,E) \to H^i(\cO_X)$ induced by the trace map 
	$tr \colon \End(E) \to \cO_X$, see \cite{HL}.
	If $\Ext^2(E,E)_0= 0$, then the moduli space is smooth at the point $[E]$.
    \smallskip
    
	In particular, if $S$ is a smooth surface and $L$ is a line bundle on $S$, then $\underline{c}$ is identified by the choice of $c_2$ so we can write $\cM^s_H(r,L,\underline{c})=\cM^s_H(r,L,c_2)$, for brevity. If $[E] $ is the isomorphism class of a $\mu_H$-stable vector bundle in $\cM^s_H(r,L,c_2)$, then 
	\begin{multline}
		\label{EQ:Edim}
		\edim(\cM^s_H(r,L,c_2)):=\dim \Ext^1(E,E)_0 - \dim \Ext^2(E,E)_0  =\\
		=2rc_2-(r-1)L^2-(r^2-1) \chi(\cO_S),
	\end{multline}
	and it is the {\it expected dimension} \cite[Def. 4.5.6]{HL} of the moduli space $\cM^s_H(r,L,c_2)$ at $[E]$.
    \smallskip

	Finally, we define the discriminant  
	\begin{equation}
	\Delta(r,L,c_2):= 2rc_2- (r-1)L^2.
	\end{equation}
	By Bogomolov's inequality the moduli space $\cM^s_H(r,L,c_2)$ is empty if $\Delta(r,L,c_2)$ is negative. If $\Delta(r,L,c_2)>>0$, the moduli space $\cM^s_H(r,L,c_2)$ is a normal, generically smooth, irreducible quasi-projective variety of the expected dimension;
   this result is due to many authors, see  \cite{MO09} for a survey.
    Moreover, when $S$ is a K3 surface, by the seminal works of Mukai (see \cite{Muk84} \cite{Muk87}), then $\cM^s_H(r,L,c_2)$, if nonempty, is a  smooth quasi-projective variety of the expected dimension  which   has a simplectic structure. 	

%-------------------------------------------------------
%-------------------------------------------------------

	\subsection{Globally generated vector bundles of rank $r$ with $r+1$ global sections } 
	Let $(X,H)$  a pair as above. Let $E$ be a vector bundle on $X$ with rank $r \geq 2 $. The {\it evaluation map} of global sections of $E$ associated to $E$ is
	\begin{equation} \ev_E \colon H^0(E) \otimes \cO_X \to E, \quad s \mapsto s(x).
	\end{equation}
	We can construct the maps
	\begin{equation} \wedge^r(\ev_E) \colon (\wedge^rH^0(E)) \otimes \cO_X \to \wedge^rE,
		\quad  s_1 \wedge s_2 \wedge \cdots \wedge s_r \to s_1(x) \wedge s_2(x) \wedge \cdots \wedge s_r(x);
	\end{equation}
	and the {\it determinant map} of $E$, namely 
	\begin{equation}
		d_E=H^0(\wedge^r \ev_E): \wedge^r H^0(E) \to H^0(\det(E)),
	\end{equation}
	i.e. the map induced by $\wedge^r(ev_E)$ on global sections.
	\smallskip
	
	We recall that $E$ is said {\it globally generated} if the  evaluation map
	$\ev_E$ is surjective. In this case, as the trivial vector bundle 
	$H^0(E) \otimes \cO_X$  is $\mu_H$-semistable, for any ample  line bundle $H$ on $X$, we obtain that $\mu_H(E) \geq 0$.
	\smallskip

    Now we assume that $E$ is globally generated and $h^0(E)= r+1$. We set $L=\det(E)$ for brevity, and we consider the exact sequence
	\begin{equation}
		\label{SES:evE}
		0 \to L^* \xrightarrow{} H^0(E) \otimes \cO_X \xrightarrow{\ev_{E}} E \to 0
	\end{equation}
	and its dual
	\begin{equation}
		\label{SES:evE_Duale}
		0 \to E^* \xrightarrow{\ev_E^*} H^0(E)^* \otimes \cO_X \xrightarrow{\gamma} L \to 0
	\end{equation}
	where $\gamma$ is the dual of the inclusion $L^*\hookrightarrow H^0(E)\otimes \cO_X$ composed via the canonical isomorphism $L\simeq (L^{*})^{*}$.
	The following is a technical result we will use in the sequel. 
	
	\begin{proposition}\label{prop: injectivitydE}
		Let $E$ be a globally generated vector bundle of rank $r\geq 2$ with $h^0(E)=r+1$. If $E$ is $\mu_H$-stable, for a ample line bundle $H$ on $X$, then 
		\begin{enumerate}[(a)]
			\item $d_E$ is injective;
			\item $\Ima(d_E)$ is equal to $\Ima(H^0(\gamma))$.
		\end{enumerate}
	\end{proposition}
	
	\begin{proof}
		(a) As the  sequence \eqref{SES:evE} is an exact sequence of vector bundles, we have an induced sequence
		\begin{equation}
			\label{SES:wedger_evE}
			0 \to \ker(\wedge^r\ev_E) \xrightarrow{\,} \bigwedge^r H^0(E) \otimes \cO_X \xrightarrow{\wedge^r\ev_E} L \to 0.
		\end{equation}
		and a canonical isomorphism $$\ker(\wedge^r\ev_E)\simeq L^{-1}\otimes \bigwedge^{r-1}E \simeq E^*$$
        which follows from the isomorphism $\bigwedge^r E=\det(E)=L$ (see \cite[Chapter II.5]{Har77}). 
        Since $E$ is $\mu_H-$stable and $\mu_H(E) \geq 0$, one can prove that $\Hom(E,\mathcal O_X) \simeq H^0(E^*)= 0$.
		So  we can conclude that the map  induced  in cohomology $$d_E=H^0(\wedge^r\ev_E) : \bigwedge^r H^0(E) \to H^0(L)$$   is injective.
	\medskip
	
	(b) Being $h^0(E)=r+1$, we have the canonical isomorphism 
	\begin{equation}
		\eta:\wedge^rH^0(E)\to H^0(E)^*\qquad \omega\mapsto \{s\mapsto \omega\wedge s\} 
	\end{equation}
	and then an isomorphism $\eta'=\eta\otimes \id_{\cO_X}: \bigwedge^r H^0(E)\otimes \cO_X \to H^0(E)^*\otimes \cO_X$.
	Consider the following exact sequences:
	\begin{equation}
		\label{SES:partial_diag}
		\begin{tikzcd}
			0 & {\Ker(\wedge^r\ev_E)} & {\bigwedge^rH^0(E)\otimes \mathcal O_X} & L & 0 \\
			0 & {\Ker(\gamma)} & {H^0(E)^*\otimes \mathcal O_X} & L & 0
			\arrow[from=1-1, to=1-2]
			\arrow[from=1-2, to=1-3]
			%\arrow["{{{\iota_1}}}"', hook, from=1-2, to=2-2]
			\arrow["{{\wedge^r\ev_E}}", from=1-3, to=1-4]
			\arrow["{\eta'}"', hook, from=1-3, to=2-3]
			\arrow[from=1-4, to=1-5]
			%\arrow["{{{\iota_3}}}"', hook, from=1-4, to=2-4]
			\arrow[from=2-1, to=2-2]
			\arrow[from=2-2, to=2-3]
			\arrow["{\gamma}", from=2-3, to=2-4]
			\arrow[from=2-4, to=2-5]
		\end{tikzcd}
	\end{equation}
	
	We claim that $\eta'(\Ker(\wedge^r\ev_E))=\Ker(\gamma)$. Recall that, for all $x\in X$, one has 
	$$L^*_x\simeq \Ker(\ev_E)_x=H^0(E\otimes \cI_x)\otimes \cO_x = \langle \tau_x \rangle\otimes \cO_x,$$
	by the short exact sequence \eqref{SES:evE}. Under our assumptions we have
	$$\Ker(\wedge^r\ev_E)_x=\{ \omega \in \wedge^r H^0(E)\,|\, \omega\wedge \tau_x=0 \}\otimes \cO_x.$$
	On the other hand, one has
	$$\gamma_x:H^0(E)^*\otimes \cO_x\to L_x$$
	is the map induced by the restriction of forms on $H^0(E)$ to $H^0(E\otimes \cI_x)$.
	Then 
	$$\Ker(\gamma)_x=\{\varphi\in H^0(E)^* \,|\, H^0(E\otimes \cI_x)\subseteq \Ker(\varphi)\}\otimes \cO_x.$$
	
	Notice that if $\omega\in \wedge^rH^0(E)$, then $\tau_x\in \Ker(\eta(\omega))\Longleftrightarrow \omega\wedge \tau_x=0$
	so 
	$$\eta'(\Ker(\wedge^r\ev_E))=\Ker(\gamma)$$
	as claimed. Then there exists $\alpha:L\to L$ which makes commutative the diagram on the right in \eqref{SES:partial_diag}. Actually, being $\eta'$ an isomorphism, by Snake Lemma, $\alpha$ is an isomorphism too. Since $L$ is a line bundle, this is an homothety.
	\medskip
	
	Finally, we have a commutative diagram
	\begin{equation}
		\begin{tikzcd}
			{\bigwedge^rH^0(E)} & {H^0(L)} \\
			{H^0(E)^*} & {H^0(L)}
			\arrow["{{d_E}}", from=1-1, to=1-2]
			\arrow["{\eta}"', hook, from=1-1, to=2-1]
			\arrow["{{{H^0(\alpha)}}}", hook, from=1-2, to=2-2]
			\arrow["{H^0(\gamma)}"', from=2-1, to=2-2]
		\end{tikzcd}
	\end{equation}
	As $H^0(\alpha)=\lambda \cdot \id_{H^0(L)}$, this concludes the proof.

\end{proof}

\section{Main construction} 

In this section, we consider a smooth complex projective variety $X$ of dimension $n\geq 2$. We recall that if $L$ is a line bundle on $X$ and $W$ is a (non-trivial) subspace of $H^0(L)$ one denotes by 
$$\varphi_{|W|} \colon X \dashrightarrow \bP(W)^*\qquad p\mapsto \{s\in W \,|\, s(p)= 0\}$$
the usual map induced by  global sections of $W$. We will simply write $\varphi_L$ instead of $\varphi_{H^0(L)}$, for brevity. 
\smallskip

\begin{definition}
	\label{DEF:AdmissibleTriple}
    Consider the collection $(X,L,H,r)$ where $X$ is a smooth complex projective variety of dimension $n\geq 2$, $L$ and $H$ are line bundles on $X$ and $r$ is an integer with  $r\geq 2$, satisfying the following conditions:
	\begin{description}
		\item[$A_1$]\label{ASS:A1} $H$ is ample and there exists a smooth irreducible curve $C$ of genus $g\geq 2$ which is complete intersection of divisors in $|H|$;
		\item[$A_2$]\label{ASS:A2} $L$ is big, nef, globally generated, $r\geq \dim(\varphi_L(X))$ and the restriction map of global sections $\rho \colon H^0(L) \to H^0(L|_{C})$ is surjective;
		%\item[$A_3$]\label{ASS:A3} one has $\deg(L|_{C})= rg + 1$;
		%\item[$A_3'$]\label{ASS:A3} if $d=\deg(L|_{C})>2g$ then $\max(d/2,d-2g)\leq r\leq d-g-1$.        
        \item[$A_3$]\label{ASS:A3} If we set $d=\deg(L|_C)$, then either
        \begin{description}
            \item [$A_3(1)$] $d=rg+1$ or;
            \item [$A_3(2)$] $r+g+1\leq d\leq \min(2r, r+2g)$
            and if $d= 2r$, $C$ is not hyperelliptic.             %$\max(d/2,d-2g)\leq r\leq d-g-1$.
        \end{description}
	\end{description}
    We will say that $(X,L,H,r)$ is \emph{admissible} if assumptions $A_1, A_2$ and $A_3$ hold.
\end{definition}

\begin{remark}
We stress that, as we are assuming $r\geq 2$ and $g\geq 2$, it does not exists $d$ that satisfies the two numerical conditions in $A_3(1)$ and $A_3(2)$ simultaneously.
\end{remark}

\begin{remark}
As will be clear in the sequel, the curve $C$ will only be auxiliary to the construction and the results will not depend on the specific choice of $C$. For this reason, $C$ is not part of the building data $(X,L,H,r)$.
\end{remark}

\hfill\par
For any $k \geq 1$, $\Gr(k,H^0(L))$  will denote the  Grassmannian variety parametrizing $k$-dimensional linear susbspaces of $H^0(L)$.

% ----

\begin{lemma}
	\label{lemmatecnico1}
	Let $(X,L,H,r)$ an admissible collection and let $C$ and $\rho$ be as in the Definition \ref{DEF:AdmissibleTriple}. Then $\rho$ induces a rational surjective map
    $$ R_C \colon \Gr(r+1, H^0(L)) \dashrightarrow\Gr(r+1,H^0(L|_{C})), \quad W \mapsto \rho(W).$$
    
     Moreover, for $W$ general in $\Gr(r+1,H^0(L))$, $|W|$ and $|\rho(W)|$ are base points free linear systems.
		% \item{} {\color{blue} if $A_3(1)$ holds}, the multiplication map of global sections 
		% $$\mu_{\rho(W)} \colon \rho(W) \otimes H^0(\omega_C) \to H^0(\omega_C \otimes L|_{C})$$
		% is an isomorphism. 
\end{lemma}

\begin{proof}
	By $A_3$ it follows that $\deg(L|_{C}) = d  \geq 2g+1$, so we have $h^1(L|_{C}) = 0$ and $h^0(L|_{C}) = d + 1 -g$.
	By $A_2$, the restriction map $\rho :H^0(L) \to H^0(L|_{C})$
	is surjective, so
	\begin{equation}
		h^0(L) \geq h^0(L|_{C})=d+1-g>r+1
	\end{equation}
    by assumptions $A_1$ and $A_3$.
    \medskip

    In particular, $\Gr(r+1,H^0(L))$ and $\Gr(r+1,H^0(L|_{C}))$ are both not empty and $$\codim_{H^0(L)}(\Ker(\rho))=h^0(L|_C)>r+1.$$
    Hence, for $W\in \Gr(r+1,H^0(L))$ general, we have that $\Ker(\rho) \cap  W = \{0\}$  so  $\rho|_{W} \colon W \to \rho(W)$ is an isomorphism. 
    \smallskip
    
    This defines the rational map $R_C$ which is also surjective since $\rho$ is surjective and by $A_1$.
	\medskip
	
We claim now that there exists a non-empty open subset of $\Gr(r+1,H^0(L))$ which parametrises base point free linear systems. Recall that there exists a canonical isomorphism
	$$\alpha:\Gr(r+1,H^0(L))\to \Gr(h^0(L)-(r+1),H^0(L)^*)$$
	which associates to $W$ the kernel $\Lambda$ of the dual of the inclusion $W\hookrightarrow H^0(L)$. Moreover, if $\Lambda=\alpha(W)$, the projection $\pi:\bP(H^0(L))^*\dashrightarrow \bP W^*\simeq \bP^r$ from $\bP(\Lambda)$ fits into the diagram
	\begin{equation}
		\xymatrix
		{
			X \ar@{->}[r]^-{{\varphi}_L}  \ar@{-->}[rd]_-{\varphi_{|W|}}  &  {\mathbb P}H^0(L)^* 
			\ar@{-->>}[d]^-{\pi}\\ & {\mathbb P}W^* }
	\end{equation}
	As $L$ is globally generated, one has that $\varphi_{|W|}$ is a morphism if and only if $ \varphi_L(X) \cap \bP(\Lambda)$ is empty. 
	Actually, this occurs for general $\Lambda \in \Gr(h^0(L)-r-1,H^0(L)^*)$ since $\codim_{\bP(H^0(L))^*}(\bP(\Lambda))=r+1\geq \dim(\varphi_L(X))+1$ by $A_2$. Hence, for general $W$, one has that $|W|$ is base point free. The elements of the linear system $|\rho(W)|$ are the intersection of the divisors in $|W|$ with $C$, so $|\rho(W)|$ is base point free too.
    \medskip
\end{proof}

\begin{remark}
	\label{REM:bound}
	By Lemma \ref{lemmatecnico1}, since as observed in the above proof, one has $h^0(L|_{C})= d+1-g$, it follows that:
	$$\dim \Gr(r+1,H^0(L)) \geq  \dim \Gr(r+1,H^0(L|_{C}))=(r+1)(d-g-r).$$
    If assumptions $A_3(1)$ holds, one has $\dim \Gr(r+1,H^0(L|_{C}))=(r^2-1)(g-1)$.
\end{remark}
\medskip

Let $W \in \Gr(r+1,H^0(L))$ such that $|W|$ is base point free. Hence, the evaluation map $\ev_W$ associated to $W$ is surjective and its kernel is a locally free sheaf on $X$ of rank $r$ which fits in the following exact sequence
\begin{equation}
	\label{evW}
	0 \to \ker(ev_W) \to W \otimes \cO_X \xrightarrow{\ev_W}  L \to 0,
\end{equation}
whose  dual is
\begin{equation}
	\label{dualevW}
	0 \to L^* \to W^* \otimes \mathcal \cO_X \to {\ker(ev_W)} ^* \to 0.
\end{equation}
We define
\begin{equation}
	\label{defEW}
	E_W \colon = {\ker(ev_W)}^*.   
\end{equation}

\begin{lemma}
	\label{lemmatecnico2}
	Let $W \in \Gr(r+1,H^0(L))$ such that $|W|$ is base point free. Then $E_W$ is a vector bundle on $X$ with the following properties:
	\hfill\par
	\begin{enumerate}[(a)]
		\item{} $\rk E_W= r$, $\det (E_W)= L$ and $c_k(E_W) = c_1(L)^k$ for $k=1,\dots, n$;
		%\item{} the Hilbert polynomial of $E_W$ is \(P_H(E_W)=(r+1)P_H(\cO_X)-P_H(L^*)\); 
		\item{} $H^0(E_W) \simeq W^*$ and $E_W$ is globally generated;
		\item{} $W=\Ima(H^0(\gamma))$ where $\gamma$ is defined in exact sequence \eqref{SES:evE_Duale};
		%\item{} $h^0(E_W \otimes E_W^*) = 1$.
	\end{enumerate}
\end{lemma}

\begin{proof} 
	In order to prove claim $(a)$, recall that $\dim(W)=r+1$ so that $\rk E_W= r$ and $\det (E_W)= L$ by the exact sequence \eqref{dualevW}. From the same sequence, one has
	$$1=c(W^*\otimes \cO_X)=c(L^*)c(E_W)=(1-c_1(L))c(E_W)=1+\sum_{k=1}^n(c_k(E_W)-c_1(L)c_{k-1}(E_W))$$
	by Whitney's sum formula. Then, by induction, one has $c_k(E_W)=c_1(L)^{k}$.
	\smallskip
	
	For claim $(b)$, we get the exact sequence 
	$$0 \to H^0(L^*) \to W^* \to H^0(E_W) \to H^1(L^*) \to \cdots$$
	passing to cohomology from the Exact Sequence \eqref{dualevW}.
	\smallskip

    Since $L$ is big and nef  by $A_2$ we have $H^q(L^*) = 0$ for $q < n$ by Kawamata-Viehweg vanishing Theorem, which implies  $H^0(E_W) \simeq W^*$. Moreover,  the composition of  the map  $W^* \otimes \cO_X \to E_W$ in exact sequence \eqref{dualevW} with the isomorphism  $W^*  \otimes \cO_X \simeq H^0(E_W) \otimes \cO_X$ is actually the evaluation map 
	$\ev_{ H^0(E_W)}$. This implies that $E_W$ is globally generated.
	\smallskip

	In order to prove claim $(c)$, start by dualizing Exact sequence \eqref{dualevW} and use what we observed in $(b)$. One gets the commutative diagram
	\[\begin{tikzcd}
		0 & {{E_W}^*} & {W^{**}\otimes \mathcal O_X} & L & 0 \\
		&& {H^0(E)^*\otimes \mathcal O_X}
		\arrow[from=1-1, to=1-2]
		\arrow[from=1-2, to=1-3]
		\arrow["{(\ev_{H^0(E)})^*}"'{pos=0.4}, from=1-2, to=2-3]
		\arrow["{\gamma'}", from=1-3, to=1-4]
		\arrow["\simeq", from=1-3, to=2-3]
		\arrow[" ", from=1-4, to=1-5]
		\arrow["\gamma"', from=2-3, to=1-4]
	\end{tikzcd}\]
	where $\gamma$ is defined in \eqref{SES:evE_Duale} while $\gamma' $ is ${\ev_W}^{**}$ composed via the canonical isomorphism $L\simeq L^{**}$. 
	
	In particular, passing to cohomology, we have that the images of $H^0(\gamma)$ and $H^0(\gamma')$ coincide. By construction we have $\Ima(H^0(\gamma'))=W$.
	\smallskip
	
	% {\color{blue} For claim (d), consider the exact sequences \eqref{evW} and \eqref{dualevW} tensored with $E_W$ and $L$, respectively: 
		% $$0 \to  E_W^* \otimes E_W \to W \otimes E_W \xrightarrow{\alpha} L \otimes E_W \to 0\qquad 0 \to \mathcal O_X \to W^* \otimes L \xrightarrow{\beta} L\otimes E_W \to 0.$$
		% Recall also that both $\alpha$ and $\beta$ are induced by evaluation maps of global sections ($\alpha$ by construction, $\beta$ as shown in the proof of Lemma \ref{lemmatecnico1} vie the isomorphism $W^*\simeq H^0(E_W)$).
		% Passing to cohomology one gets a commutative diagram
		% \[\begin{tikzcd}[sep=small]
			% 	0 & {H^0(E_W^*\otimes E_W)} & {W\otimes H^0(E_W)} && {H^1(\cO_X)} & \cdots \\
			% 	&&& {H^0(E\otimes L)} \\
			% 	0 & {H^0(\cO_X)} & {H^0(L)\otimes H^0(E_W)} && {H^1(E_W^*\otimes E_W)} & \cdots
			% 	\arrow[from=1-1, to=1-2]
			% 	\arrow[hook,from=1-2, to=1-3]
			% 	\arrow["{h^0(\alpha)}"{description}, bend right=-20,  from=1-3, to=2-4]
			% 	\arrow["\iota\otimes \id", hook, from=1-3, to=3-3]
			%         \arrow[from=1-5, to=1-6]
			% 	\arrow[bend right=-20, from=2-4, to=1-5]
			% 	\arrow[bend right=20, from=2-4, to=3-5]
			% 	\arrow[from=3-1, to=3-2]
			% 	\arrow[hook,from=3-2, to=3-3]
			% 	\arrow["{h^0(\beta)}"{description}, bend right=20, from=3-3, to=2-4]
			%         \arrow[from=3-5, to=3-6]
			% \end{tikzcd}\]
		% with exact ``rows''. Hence, $\End(E_W)\simeq H^0(E_W^* \otimes E_W) \simeq \ker(h^0(\alpha))$ is a subspace, via the inclusion $\iota\otimes \id$, of $\Ker(h^0(\beta))\simeq H^0(\cO_X)\simeq \bC$. Being $\End(E_W)$ non-trivial, one concludes. 
		% }
	
\end{proof}

\begin{remark}
	By Lemma \ref{lemmatecnico2} it follows that 
	$$\Delta_H(E_W)=(r+1)c_1(L)^2H^{n-2}>0,$$ 
	so the vector bundle $E_W$  satisfies the generalized Bogomolov's necessary condition (see Equation \ref{EQ:Discriminant}) for $\mu_H$-semistability. Actually, we will prove in Proposition \ref{prop: stableEw} that $E_W$ is $\mu_H$-stable.
\end{remark}

Let  $C$ as in Definition \ref{DEF:AdmissibleTriple}. If $W\in \Gr(r+1,H^0(L))$ is general, by Lemma \ref{lemmatecnico1} we have that $\rho(W) \in \Gr(r+1,H^0(L|_{C}))$ and, moreover, $|\rho(W)|$ is base points free. This means that the evaluation map $ev_{\rho(W)} \colon \rho(W) \otimes  {\mathcal O}_C \to L|_{C}$ is surjective.  Its kernel is a locally free sheaf on $C$ which fits in the exact sequence
\begin{equation}
	\label{evrhoW}
	0 \to \ker(ev_{\rho(W)}) \to \rho(W)\otimes {\mathcal O}_C \to  L|_{C} \to 0,
\end{equation}
whose  dual is
\begin{equation}
	\label{dualevrhoW}
	0 \to {L^*}|_{C} \to \rho(W)^* \otimes {\mathcal O}_C \to {\ker( ev_{\rho(W)})}^* \to 0.
\end{equation}
Then, by construction,
\begin{equation}
	\label{defErhoW}
	E_{\rho(W)} \colon = {\ker(ev_{\rho(W)})}^*  
\end{equation}
is a vector bundle of rank $r$, whose determinant is $\det E_{\rho(W)}= L|_{C}$.

\begin{remark}
\label{REM:GloGen}
The same argument used in Lemma \ref{lemmatecnico2} proves that $E_{\rho(W)}$ is globally generated.
\end{remark}

% \begin{lemma}
% 	\label{lemmatecnico3}
% 	Let $W \in U_C$, then  $E_{\rho(W)}$ is a vector bundle on $C$ with the following properties:
% 	\hfill\par\noindent
% 	\begin{enumerate}[(a)]
% 		\item{} $\rk(E_{\rho(W)})=r$ and $\det E_{\rho(W)}= L|_{C}$;
% 		\item{} $H^0(E_{\rho(W)}) \simeq \rho(W)^*$ and  $E_{\rho(W)}$  is globally generated;
% 		\item{} $E_{\rho(W)}$ is stable.
% 	\end{enumerate}
% \end{lemma}

\begin{lemma}
	\label{lemmatecnico3}
	Let $C$ be as in Definition \ref{DEF:AdmissibleTriple}. Then, there exists an open dense subset $U_C \subseteq \Gr(r+1,H^0(L))$ such that  for any $W\in U_C$, $\vert W \vert $ is base points free, $\rho(W) \simeq W$ and  $E_{\rho(W)}$ is a stable vector bundle.
\end{lemma}

\begin{proof}
	Let $W \in  \Gr(r+1,H^0(L))$ such that $\vert W \vert $ is base points free and $\rho(W) \simeq W$.  We distinguish two cases depending on whether $A_3(1)$ or $A_3(2)$ applies. 

\begin{itemize}
    \item Assume that $A_3(1)$ holds. We consider the exact sequence induced by  \eqref{dualevrhoW}, passing to cohomology:
	$$ 0 \to \rho(W)^* \to H^0(E_{\rho(W)}) \to H^1(L^*|_C) \to \rho(W)^* \otimes H^1(\cO_C) \to H^1(E_{\rho(W)}) \to 0.
	$$
	By $A_3(1)$, one has $\deg(E_{\rho(W)})= rg+1$ so that $\chi(E_{\rho(W)})=r+1$.
	This implies that
	$\rho(W)^* \simeq H^0(E_{\rho(W)})$ if and only if $h^1(E_{\rho(W)})=0$. This happens exactly when the map 
	$$H^1(L^*|_{C}) \to \rho(W)^* \otimes H^1({\mathcal O}_C)$$
	is an isomorphism i.e. when the dual map 
	$$m_{\rho(W)} \colon \rho(W) \otimes H^0(\omega_C) \to H^0(\omega_C \otimes L|_{C})$$
	is an isomorphism. 
    \medskip

	We claim that for general $W$, the multiplication map $m_{\rho(W)}$ is an isomorphism. Since $\deg(L|_C)=rg+1$ we have $\deg(L|_C\otimes \omega_C)=g(r+2)-1\geq 2g-1$ so
	$$h^1(L|_C\otimes \omega_C)=0\qquad \mbox{ and }\qquad h^0(L|_C\otimes \omega_C)=\chi(L|_C\otimes \omega_C)=g(r+1).$$
	By \cite{Bri}, one has that $\mu_{W'}$ is surjective for $W'$ general in $\Gr(r+1,H^0(L|_C))$ so
	$$V = \{ W' \in \Gr(r+1, H^0(L|_{C})) \,| \, m_{W'} \  \text{is  an isomorphism} \}$$
	is a dense open subset of $\Gr(r+1,H^0(L|_C))$.
	\smallskip
	As $R_C$ is a rastional surjective map, see \ref{lemmatecnico1}, 
	then, $R_C^{-1}(V)$ is a non-empty open subset of $\Gr(r+1,H^0(L))$, and $m_{\rho(W)}$ is an isomorphism for $W \in R_C^{-1}(V)$. Summing up, we concluded that for $W$ general, $H^0(E_{\rho(W)})\simeq \rho(W)^*$.
    \medskip

    In order to prove the stability of $E_{\rho(W)}$, %we use the same argument of \cite{BV02}, adapted to our case. 
    we assume that there exists a proper subbundle $G \subset E_{\rho(W)}$ of degree $d$ and rank $s \leq r-1$, such that 
	$$\mu(G) =\frac{d}{s} \geq \mu(E_{\rho(W)})=\frac{rg+1}{r}.$$ This implies that 
	$$d \geq sg+1\qquad 
	\mbox{ and } \qquad \chi(G)\geq sg+1+s(1-g)=s+1.$$ 
	In particular, we have that $h^0(G)\geq s+1$. We claim now that $h^0(G)=s+1$ and that $G$ is globally generated. 
	\smallskip
	
	Recall that $G$ is a subbundle of $E_{\rho(W)}$, which is globally generated and is such that $h^0(E_{\rho(W)})= r+1$. Assume, by contraddiction, that $h^0(G)>s+1$. Then, the sections of $H^0(G)\subseteq H^0(E_{\rho(W)})$ span a vector bundle $G'$ in $E_{\rho(W)}$ of rank at most $s$. On the other hand, the remaining sections of $H^0(E_{\rho(W)})$ cannot increase the rank of the spanned vector bundle by more than $h^0(E_{\rho(W)})-h^0(G)<r+1-(s+1)=r-s$. This is impossible as we could have that $E_{\rho(W)}$ is not globally generated: we have that $h^0(G)=s+1$.
	
	In a similar way one proves that $G$ is globally generated (since otherwise we would have points on $C$ where $s+1$ sections would span a vector space of dimension lower than $s$).
	\smallskip
	
	Being $h^0(G)=s+1$, we have 
	$$s+1-h^1(G)=\chi(G)\geq s+1$$
	so that $h^1(G)=0$ and $\deg(G)=sg+1$.
	
	The evaluation maps of $G$ and $E_{\rho(W)}$, fit in the commutative diagram
	\begin{equation}
		\xymatrix@R=10pt{
			0 \ar[r] & M^* \ar[r]\ar@{^{(}->}[d] & H^0(G)\otimes \cO_C \ar[r]\ar@{^{(}->}[d] & G \ar[r]\ar@{^{(}->}[d] & 0     \\
			0 \ar[r] & L^* \ar[r] & H^0(E_{\rho(W)})\otimes \cO_C \ar[r] & E_{\rho(W)} \ar[r] & 0     \\
		}
	\end{equation}
	Since $M^*$ is a subsheaf of $L^*$ we have
	$$\deg(M^*)=-\deg(G)=-(sg+1)\leq \deg(L^*)=-(rg+1),$$
	and thus
	$s\geq r$, which is impossible. 
    \item  Assume that $A_3(2)$ hold. 
    By the assumptions it follows immediately that $r\geq g+1$ and $d=\deg(L|_C)\geq 2g+2$ so $h^0(L|_C)=d+1-g$. We set 
    $$c:=\codim_{H^0(L|_C)}(\rho(W))=h^0(L|_C)-(r+1)=d-g-r.$$
    By our assumptions on $d$, it follows that
    $$1 \leq c \leq g \qquad \mbox{ and }\qquad d \geq 2g+2c.$$
    Then we can apply \cite[Theorem 1.3]{Mis08}: $\Ker(ev_{V})$ is stable for a general $V \subset H^0(L|_C)$  of codimension $c$, unless $d=2g+2c$ and $C$ is hyperelliptic (case which is excluded by $A_3(2)$). Since the rational map $R_C$ is surjective, see lemma \ref{lemmatecnico1}, this gives a non-empty open subset of $\Gr(r+1,H^0(L))$ such that $E_{\rho(W)}$ is stable for any $W \in U_C$.
\end{itemize}
\end{proof}

\begin{proposition}\label{prop: stableEw}
	Let $C$ be as in Definition \ref{DEF:AdmissibleTriple} and consider $W \in U_C$. Then  
	\begin{enumerate}[(a)]
		\item $E_W|_C\simeq E_{\rho(W)}$;
		\item $E_W$ is $\mu_H$-stable;
		\item the determinant map $d_{E_W}$ is injective and has image $W$.
	\end{enumerate}
\end{proposition}
\begin{proof} Let $W \in U_C$, $E_W$  defined in Equation \eqref{defEW} and $E_{\rho(W)}$ defined in Equation \eqref{defErhoW}. \hfill\par\noindent
	In order to prove claim (a), we start by tensoring \eqref{dualevW} with $\mathcal O_C$ and get the exact sequence
	\begin{equation}
		0 \to L^*|_{C} \xrightarrow{(\ev_W^*)|_C} W^* \otimes \mathcal \cO_C \to E_W|_{C} \to 0.
	\end{equation}
	By Lemma \ref{lemmatecnico1} we have that $W \simeq \rho(W)$, hence 
	$W^* \simeq \rho(W)^*$, moreover $$(\ev_W^*)|_C = (\ev_W |_C)^*$$ so  we get the commutative diagram
	\begin{equation}
		\begin{tikzcd}
			0 & {L|_C^*} & {W^*\otimes \cO_C} & {E_W|_C} & 0 \\
			0 & {(L|_C)^*} & {\rho(W)^*\otimes \cO_C} & {E_{\rho(W)}} & 0.
			\arrow[from=1-1, to=1-2]
			\arrow["{(\ev_W|C)^*}", from=1-2, to=1-3]
			\arrow["{\simeq }", from=1-2, to=2-2]
			\arrow[from=1-3, to=1-4]
			\arrow["{\rho^*\otimes\id}", from=1-3, to=2-3]
			\arrow[from=1-4, to=1-5]
			\arrow["{\simeq }", from=1-4, to=2-4]
			\arrow[from=2-1, to=2-2]
			\arrow["{\ev_{\rho(W)}^*}", from=2-2, to=2-3]
			\arrow[from=2-3, to=2-4]
			\arrow[from=2-4, to=2-5]
		\end{tikzcd}
	\end{equation}
	
	To prove claim (b), assume that there exists a subbundle $G \hookrightarrow E_W$ of rank $s \leq r-1$ such that $\mu_H(G) \geq \mu_H(E_W)$.
	Being $C$ a complete intersection of divisors in $|H|$ one has $\deg(F|_C)=c_1(F)\cdot H^{n-1}$ for any vector bundle $F$ on $X$. In particular, $G|_C$ and $E_W|_C$ are vector bundles on $C$ which satisfy 
	$$\mu(G|_C)=\frac{\deg(G|_{C})}{s}=\frac{c_1(G)\cdot H^{n-1}}{s}= \mu_H(G) \qquad$$
	$$ \mu(E_W|_{C})=\frac{\deg(E_W|_{C})}{r}=\frac{c_1(E_W)\cdot H^{n-1}}{r}= \mu_H(E_W)$$
	so that $\mu(G|_C) \geq \mu(E_W|_C)$. In particular, $E_W|_C$ is not stable. This is impossible, since $E_W|_C$ is isomorphic to $E_{\rho(W)}$, which is stable by Lemma \ref{lemmatecnico3}.
	\smallskip
	
	In order to prove claim (c), since  $E_W$ is $\mu_H$-stable, globally generated and has $h^0(E_W)= r+1$ (by Lemma \ref{lemmatecnico2}), then,  by Proposition \ref{prop: injectivitydE}, one has that the determinant map $d_E$ is injective and $\Ima (d_E)= \Ima H^0(\gamma)$.
	Hence, by Lemma \ref{lemmatecnico2}, one gets $\Ima H^0(\gamma)= W$. This concludes the proof.
\end{proof}

\begin{remark}
	The same conclusion holds when $C$ is taken to be a smooth irreducible curve of genus $g\geq 2$ such that 
	$$\deg(F|_C)=c_1(F)\cdot H^{n-1}\qquad \mbox{ for any vector bundles } F \mbox{ on } X.$$
	This is clearly true if $C$ is a complete intersection of divisors in $|H|$.
\end{remark}

\begin{definition}
	\label{DEF:definizione di U}
	Assume that $(X,L,H,r)$ is admissible. We set $U$ to be the union of all the open dense subsets $U_C$ defined in Lemma \ref{lemmatecnico3}. We also set $\underline{c}= (c_1(L), c_1(L)^2, \cdots, c_1(L)^n)$.
\end{definition}

\hfill\par
Consider now the moduli space parametrizing $\mu_H$-semistable vector bundles $E$ on $X$ with rank $r$, $\det E \simeq L$ and Chern classes $c_i(E)= c_1(L)^i$, $i= 2,\dots, n$.
\smallskip 

Let $U \subset \Gr(r+1,H^0(L))$ the open subset defined in Definition \ref{DEF:definizione di U}.
We have a map 
\begin{equation}
	%\label{DEF:Phi}
	U \to  \cM^s_H(r,L,\underline{c}) \qquad 
	W \mapsto [E_W].
\end{equation}
\begin{proposition}
	The above map defines   a rational map 
	$$\Phi :\Gr(r+1,H^0(L)) \dashrightarrow   \cM^s_H(r,L,\underline{c}).$$
	In particular, $\cM^s_H(r,L, \underline{c})$  is not empty. 
\end{proposition}

\begin{proof}
	We prove that $\Phi|_{U}$ is a morphism.
	Let $\cU$ and $\cQ$ be the universal and quotient bundle on $G:=\Gr(r+1,H^0(L))$. They fit into the exact sequence
	\begin{equation}
		0\to \cU\hookrightarrow H^0(L)\otimes \cO_G \to \cQ \to 0.
	\end{equation}
	Consider the product $G\times X$ with its projection $\pi_i$ on its factors. 
	By pulling back along $\pi_1$ one has an inclusion $\pi_1^*\cU\hookrightarrow H^0(L)\otimes \cO_{G\times X}$. On the other hand, one can also pullback the evaluation map $\ev_{H^0(L)}\otimes \cO_X\to L$ along $\pi_2$. The composition $\theta$ of these maps gives a commutative diagram
	$$
	\xymatrix{
		\pi_1^*(\cU) \ar@{^{(}->}[r]\ar[rd]_-{\theta} & H^0(L)\otimes \cO_{G\times X}\ar[d]^-{\pi_2^*(\ev)} \\
		& \pi_2^*(L)
	}
	$$
	In particular, $\theta|_{\{W\}\times X}$ is the evaluation map $\ev_W:W\otimes \cO_X\to L$.
	Hence, $\Ker(\theta)$ is locally free on the open set $U\times X$ by the above results. Then we can set 
	$$\cE=\cHom\left(\Ker(\theta)|_{U\times X},\cO_{U\times X}\right)$$
	and observe that for all $W\in U$ one has
	$$\cE|_{\{W\}\times X}\simeq \Ker(W\otimes \cO_{X}\to L)^*=E_W.$$
	This implies that the map $U \to \cM^s_H(r,L,\underline{c})$ such that $W \mapsto [E_W]$ is a morphism.
\end{proof}

We stress that, a priori, $\Phi$ could be defined on a bigger open subset containing $U$. 

\begin{theorem}
\label{thm:inj_mor_moduli}
	The restriction $\Phi \colon U \to \cM^s_H(r,L,\underline{c})$  is an injective morphism. In particular, the moduli space $ \cM^s_H(r,L,\underline{c})$ contains a variety birational to $\Gr(r+1,H^0(L))$.
\end{theorem}
\begin{proof}

	By 
	%Lemma \ref{lemmatecnico2} and 
	by Proposition \ref{prop: injectivitydE} have a map 
	$$d \colon \Phi(U)  \to \Gr(r+1,H^0(L))$$ 
	sending $E \mapsto \Ima(d_E)$, where $d_E$ is the determinant map of $E$.   By  Proposition \ref{prop: stableEw} we have
	that $d(\Phi(W)) = W$ for any $W \in U$. This implies that $\Phi$ is injective. In particular, by generic smoothness, the closure of $\Phi(U)$ is a variety birational to $\Gr(r+1,H^0(L)))$.
\end{proof}

As a immediate consequence, $\cM^s_H(r,L,\underline{c})$ has an irreducible component of dimension at least $(r+1)(d-g-r)$, by Remark \ref{REM:bound}.

%-------------------------------------------------------
%-------------------------------------------------------
%-------------------------------------------------------
%-------------------------------------------------------
%-------------------------------------------------------

\section{Some examples for surfaces}

In the previous section, we proved that, given an admissible collection $(X, L, H,r)$, then the moduli space $\mathcal M_H^s(r,L,\underline{c})$ contains a subvariety birational to the Grassmanian $\Gr(r+1, H^0(L))$. In this section, we will present some examples of such collections when $X$ is a smooth algebraic surface, denoted, from now on, by $S$. We will produce examples for every possible value of the Kodaira dimension. 
\smallskip

The admissible data $(S,L,H,r)$ will be presented with more details for the case of surfaces of general type and for surfaces of Kodaira dimension $0$; for the other cases, we will not report most of the computations since they are similar to previous ones. Moreover, for the case of $K3$ surfaces, we will make a finer analysis and obtain Lagrangian subvarieties of the moduli space of sheaves (with suitable invariants).

\subsection{Surfaces of general type}
Let us assume that $S$ is a minimal surface of general type with $K_S$ ample. In particular, under these assumptions, $S$ coincides with its canonical model. We also assume that $S$ admits a smooth irreducible curve $C$ in $|K_S|$. Notice that its genus is $g(C)=1+K_S^2\geq 2$.
\smallskip

We observe that $mK_S$ is ample for all $m\geq 1$. Nevertheless, it is not necessarily globally generated. 
\begin{remark}
    \label{REM:BombieriReider}
    As $S$ is a canonical model, then, by results of Bombieri and Reider (see \cite{Bom73} and \cite{Rei88}), one has that $mK_S$ is very ample (and so also globally generated) when
    \begin{equation}
		\label{EQ:Bound_m2_fede}
		m\geq 5\, \mbox{ if }\, K_S^2\leq 2\qquad \mbox{ or }\qquad m\geq 3 \, \mbox{ if }\, K_S^2\geq 3.
	\end{equation}
\end{remark}

\begin{remark}
\label{rem: m_geq_3_implies_A2}
If we assume $m\geq 3$, one has that $(m-1)K_S$ and $(m-2)K_S$ are ample. Hence, for any $j\geq 1$, one has
$$H^j(mK_S)=H^j(K_S+(m-1)K_S)=0\qquad H^j(mK_S-C))=H^j(K_S+(m-2)K_S)=0$$
by the Kodaira Vanishing Theorem. In particular, 
we have $H^1(mK_S-C)=0$, so the restriction map $\rho:H^0(mK_S)\to H^0(mK_S|_C)$ is surjective. Moreover we have
$$h^0(mK_S)=\chi(mK_S)=\chi(\cO_S)+\frac{m(m-1)}{2}K_S^2.$$
\end{remark}
We set $H:=K_S$ and $L:=mK_S$,   we want to find positive integers $m$ and $r\geq 2$ such that $(S, mK_S, K_S, r)$ is admissible and then Theorem \ref{thm:inj_mor_moduli} applies. 
\smallskip

We observe that property $A_1$ is automatically satisfied by our assumptions on $S$. Before investigating property $A_2$, let us study the numerical properties $A_3$. Let us study separately what are the constraints on $m$ and $r$ for which either $A_3(1)$ or $A_3(2)$ holds, since our construction cannot be carried out for all pairs $(r,m)$.
\medskip

\textbf{$A_3(1)$:} The condition  $\deg(L|_C)=rg(C)+1$ holds if and only if
\begin{equation}
	\label{eq: C=KS_case_fede}
	m K_S^2=r (K_S^2+1)+1.
\end{equation}
\begin{lemma}
	\label{LEM:m and r_fede}
	Property $A_3(1)$ holds if and only if $r$ and $m$ satisfy one of the following necessary conditions:
	\begin{description}
		\item[$K_S^2=1$] $r\geq 2$ and $m=2r+1$;
		\item[$K_S^2=2$] $r=2a-1$ and $m=r+a=3a-1$ with $a\geq 2$.    
		\item[$K_S^2\geq 3$] $r=a K_S^2-1$ and $m=r+a=a(K_S^2+1)-1$ with $a\geq 1$.
	\end{description}
\end{lemma}
\begin{proof}
	If $K_S^2=1$ the condition $m=2r+1$ follows directly from Equation \eqref{eq: C=KS_case_fede}. 
	\smallskip
	
	Assume now $K_S^2\geq 2$. Reducing Equation \eqref{eq: C=KS_case_fede} modulo $K_S^2$ one gets $r\equiv K_S^2-1 \mod K_S^2$, so that $r=aK_S^2-1$ for suitable $a\geq 1$. Then 
	\begin{equation}
		\label{EQ:mFROMr}
		m=\frac{1}{K_S^2}\left[r (K_S^2+1)+1\right]=\frac{1}{K_S^2}\left[(aK_S^2-1)(K_S^2+1)+1\right]=aK_S^2+(a-1)=r+a.
	\end{equation}
	If $K_S^2=2$ one has $r=2a-1$ which satisfy the constrain $r\geq 2$ only if $a\geq 2$, so we have to exclude the case $a=1$.
\end{proof}

$A_3(2):$ In this case, the condition is 
\begin{equation}
\label{EQ:A32_SGT}
r+K^2_S+2\leq mK_S^2\leq \min(2r, r+2K_S^2+2), \quad \makebox{and \ if} \quad mK_S^2=2r, \ C \makebox{ is \ not hyperlliptic}.
\end{equation}

For brevity, we consider the set 
$$\cS_{\bar{r}}=\bigcup_{r\geq \bar{r}}\{(r,m(r)),(r,m(r)+1)\,|\, r\equiv -2 \mod K_S^2\}\cup \{(r,\lceil m(r)\rceil)\,|\, r\not\equiv -2 \mod K_S^2\}$$
where we define $m(r)=1+\frac{r+2}{K_S^2}$.

\begin{lemma}
\label{lem: Propr_A3(2)_surf_gen_type}
    Condition $A_3(2)$ holds if and only if the pair $(r,m)$ falls in one of the following cases
    \begin{equation*}
    \begin{array}{c |c|c}
        K_S^2 & \mbox{sporadic pairs} & \mbox{standard pairs}\\[0.1cm] \hline \hline
        1 & (4,7) & \cS_5\\ \hline
        2 & (4,4)^{\dagger},(5,5)^{\dagger},(6,6)^{\dagger},(6,5) & \cS_7\\ \hline
        3 & (6,4)^{\dagger},(7,4) & \cS_8\\ \hline
        4 & (6,3)^{\dagger},(8,4)^{\dagger},(10,5)^{\dagger},(9,4),(10,4) & \cS_{11} \\ \hline    
        K_S^2\geq 5 \mbox{ odd} & (4,2K_S^2)^{\dagger}, \{(r,3) \,\,|\,\, \lceil 3K_S^2/2\rceil<r\leq2K_S^2-2 \} & \cS_{2K_S^2+1} \\ \hline
        K_S^2\geq 5 \mbox{ even} & (3,3K_S^2/2)^{\dagger}, (4,2K_S^2)^{\dagger}, \{(r,3) \,\,|\,\,  3K_S^2/2 <r\leq2K_S^2-2 \} & \cS_{2K_S^2+1}    \end{array}
    \end{equation*}
    For those pairs marked with the symbol $\dagger$, we also require that the general curve $C\in |H|$ is not hyperelliptic. 
\end{lemma}

\begin{proof}
The pairs in the table are obtained by analysing the condition \eqref{EQ:A32_SGT}. When $K_S^2=1$, the pairs $(3,6)$ and $(4,8)$ were excluded since, the general element of $|K_S|$ is hyperelliptic.
\end{proof}

We can finally state and prove
\begin{theorem}
	\label{THM:MainTHM for SGT}
	Let $S$ be a minimal surface of general type with ample canonical class. Assume that the pair $(r,m)$ satisfies either the condition of Lemma \ref{LEM:m and r_fede} or Lemma \ref{lem: Propr_A3(2)_surf_gen_type} and that the canonical linear system $\vert K_S\vert$ contains a smooth irreducible curve. Then, with the possible exception of the case $K_S^2=2$ with $(r,m)=(4,4)$, $(S,mK_S,K_S,r)$ is admissible and there exists a subvariety of $\cM_{K_S}^s(r,mK_S, m^2K^2_S)$ birational to $\Gr(r+1,H^0(mK_S))$. 
\end{theorem}
\begin{proof}
    As remarked above, property $A_1$ is automatically satisfied by assumptions on $S$. Instead, property $A_3$ holds by Lemmas \ref{LEM:m and r_fede} or \ref{lem: Propr_A3(2)_surf_gen_type}. Finally, for all the above pairs with the exception, for $K_S^2=2$, of the pair $(4,4)$, the Inequalities  
	\eqref{EQ:Bound_m2_fede} in Remark \ref{REM:BombieriReider} hold so $mK_S$ is very ample and $m\geq 3$. In particular, condition $A_2$ is automatically satisfied from Remark \ref{rem: m_geq_3_implies_A2}. Then the claim follows by applying Theorem \ref{thm:inj_mor_moduli}.
\end{proof}
In order to apply Theorem \ref{THM:MainTHM for SGT} we need to check whether $|K_S|$ actually contains a smooth and irreducible element. Unfortunately, the existence of such a curve really depends on the family of surfaces we are considering and not only on numerical data on $S$.
\smallskip

If we know that $\vert K_S\vert$ is base-point free, then the problem is solved by Bertini's Theorem. However, the assumption we need, namely to pick up a smooth irreducible curve $C$ in $|K_S|$, is weaker than to require $|K_S|$ base-point free. 

% Unluckily, there are no numerical conditions on $K_S^2$ which guarantee that the canonical map $\Phi_{K_S}$ is a morphism\footnote{On the contrary, this happens for the bicanonical map, which is always a morphism if $K^2_S\geq 5$, for example.}. Indeed, in \cite{Pi12}, the author constructed minimal surfaces of general type $S$ having non-trivial fixed part of $|K_S|$ and $K^2_S$ arbitrarily high. 
\smallskip

Indeed, there are several examples of minimal surfaces of general type with ample canonical system with base points, but with a smooth irreducible canonical curve. For example, if $K_S^2=1$, one has Todorov's surfaces (see \cite{Tod80}) for which $|K_S|$ has fixed part and some surfaces studied by Horikawa and Kodaira (see \cite{Hor76}) for which $|K_S|$ has a single (simple) base point. 

\begin{remark}
	%Let us consider a minimal surface $S$ of general type with an ample canonical class. 
	By Reider's Theorem (see \cite{Rei88}), the bicanonical map is always a morphism if $K^2_S\geq 5$. This implies that the general bicanonical curve $C$ is smooth and irreducible by Bertini. 
	
	Thus, it is natural to construct other examples by setting $H:=2K_S$ and $L:=mK_S$ as property $A_1$ always holds. Using similar computations to satisfy condition $A_3(1)$ as in the previous case, we obtain the following theorem. 
\end{remark}
\begin{theorem}
	\label{THM:MainTHM2Ks for SGT}
	Let $S$ be a minimal surface of general type with a ample canonical class and $K_S^2\geq 6$, with $K_S^2$ even number. Given an even number $a\geq 2$, let us consider a pair $(r,m)$ such that 
	$$r=a K_S^2-1\qquad\mbox{ and }\qquad m=\frac{a(3K_S^2+1)-3}{2}.$$
	Then $(S,mK_S,2K_S,r)$ is admissible and there exists a subvariety of $\cM_{2K_S}^s(r,mK_S, m^2K^2_S)$ birational to $\Gr(r+1,H^0(mK_S))$. 
\end{theorem}
We point out that one can also try to get constrains on $r$ and $m$ in order to satisfy condition $A_3(2)$ instead of $A_3(1)$. In this case, we would get a similar theorem such as Theorem \ref{THM:MainTHM2Ks for SGT}.

\subsection{Surfaces with Kodaira dimension $0$} 
Let us assume now that $S$ is a smooth algebraic surface with $K_S\equiv_{num}0$. Let us consider a very ample line bundle $H$ on $S$. The assumption on the very ampleness of $H$ puts a lower bound on the possible values of $H^2$, depending on the class of surfaces we are considering. For example, if $S$ is a K3, one has $H^2\geq 4$ with equality realised if and only if $S$ is a smooth quartic in $\bP^3$. If $S$ is not a K3, one necessarily has $H^2\geq 10$ (see \cite{Mum}, \cite{BPV}, \cite{CDL}, for example). We recall, moreover, that $H^2$ is even since $K_S$ is numerically trivial.
\smallskip

By assumption, a general curve $C \in \vert H \vert $ is smooth and irreducible of genus $g(C)= 1+\frac{1}{2}H^2 \geq 2$, hence property $A_1$ is satisfied.
We set  $L:=mH$, with $m \geq 2$. By Kodaira-vanishing, we also have that property $A_2$ holds.
\smallskip

Let's consider condition $A_3(1)$. As we set $L=mH$, the condition can be rewritten as 
\begin{equation}
	\label{eq: C=KS_Ksnumtrivial}
	mH^2= r\left(1+ \frac{1}{2}H^2\right)+1.
\end{equation}
It is easy to see that this condition implies that $H^2$ has to be a multiple of $4$.
\smallskip

With similar computations as the ones done for surfaces of general type, we obtain the following result.

\begin{lemma}
	\label{LEM:CondA31_KSnumtrivial}
    For brevity, we set $h:=H^2/4$ with $H$ very ample as above. The numerical condition $A_3(1)$ holds if and only if $r$ and $m$ can be written, for a given natural number $a\geq 1$ as follows: 
	\[
	\left.\begin{aligned}
		r &= 4a+1      &\qquad \mbox{and}\qquad& m = 3a+1,&      &\qquad \makebox{if } h=1, \\[3pt]
		r &= 4ha-2h-1 &\qquad \mbox{and}\qquad& m = (1+2h)a-h-1,& &\qquad \makebox{if } h\geq 2.
	\end{aligned}\right.
	\]
\end{lemma}

Similarly, the condition $A_3(2)$ can be rewritten as 
\begin{equation}
\label{EQ:A32_KSnumtrivial}
r+\frac{1}{2}H^2+2\leq mH^2\leq \min(2r, r+H^2+2), \quad \makebox{and \ if} \quad mH^2=2r, \ C \makebox{ is \ not hyperlliptic}.
\end{equation}

We define an auxiliary set in order to describe in a more compact way the set of solutions. For a given $h\geq 2$ consider the following all $m,h\geq 2$ set
$$a_m=h(2m-2)-2\qquad b_m=a_m+h = h(2m-1)-2$$
and observe that $a_{m+1}-b_m=h$ so that the intervals $[a_m,b_m]$ are all disjoint (and exactly $h+1$ integer can be found in any of these intervals). In analogy with what we have done for the case of surfaces of general type, we consider the set 
$$\cT_{\bar{m}}=\bigcup_{m\geq \bar{m}}\big([a_m,b_m]\cap \bZ\big)\times \{m\}.$$

\begin{lemma}
	\label{LEM:CondA32_KSnumtrivial}
    For brevity, we set $h:=H^2/2$ where $H$ is a very ample divisor. Then, the numerical condition $A_3(2)$ holds if and only if the pair $(r,m)$ falls in one of the following cases:
    \begin{equation*}
    \begin{array}{c |c|c}
        h & \mbox{sporadic pairs} & \mbox{standard pairs}\\[0.1cm] \hline \hline
        2 & (4,2),(6,3),(7,3),(8,3) & \cT_4\\ \hline
        3 & (6,2),(7,2) & \cT_3\\ \hline
        4 & (8,2),(9,2),(10,2)  & \cT_3\\ \hline 
        \geq 5 & (2h,2)^{\dagger},(2h+1,2), ([2h+2,b_2]\cap \bZ)\times \{2\}  & \cT_3\\
        \end{array}
    \end{equation*}
    For those pairs marked with the symbol $\dagger$, if $K_S$ is not trivial, we also require that the general curve $C\in |H|$ is not hyperelliptic. 
\end{lemma}

\begin{proof}
The pairs are obtained by analysing the condition \eqref{EQ:A32_KSnumtrivial}. The first value of $h$ to be considered is $2$ since the minimum value of $H^2=2h$ for a very ample divisor on a surface with numerically trivial canonical bundle is $4$. The condition $mH^2=2r$ (which is the case for which we need to check whether the general element of $|H|$ is not hyperelliptic) is satisfied only by the pairs
$$(6,3) \mbox{ for } h=2\quad \mbox{ and }\quad  (2h,2) \mbox{ for } h\geq 2.$$
We claim that the only possible cases, among those, for which one could have that the general element of $|H|$ is hyperelliptic, appear at most for $h\geq 5$ and when $K_S$ is not trivial.
Indeed, if $K_S$ is trivial and if $C$ is a smooth element in $|H|$, the canonical divisor of $C$ is given by $K_C=H|_C$. This implies that the restriction of the embedding $\varphi_{|H|}$ induces an embedding of $C$ given by a subsystem of the canonical system. Then, $C$ cannot be hyperelliptic when $K_S$ is trivial.
In order to conclude the proof, it is enough to remember, as recalled at the beginning of the subsection, that $h\geq 5$ if $S$ is not a K3.
\end{proof}

\begin{theorem}
	\label{THM:MainTHM2Ks for KodDim0}
	Let us consider a surface $S$ with $K_S\equiv_{num} 0$ and let $H$ be a very ample line bundle.  Assume that $(r,m)$ satisfies either the conditions  of Lemma \ref{LEM:CondA31_KSnumtrivial} and \ref{LEM:CondA32_KSnumtrivial}. 
	Then $(S,mH,H,r)$ is admissible and there exists a subvariety of $\cM_{H}^s(r,mH,(mH)^2)$ birational to $\Gr(r+1,H^0(mH))$. 
\end{theorem}

It is actually useful to be more precise about the subvariety of $\cM_{H}^s(r,L,L^2)$ which is birational to the Grassmannian $\Gr(r+1,H^0(L))$ in these cases. Indeed, if we can apply Theorem \ref{thm:inj_mor_moduli}, then, we have an injective morphism $$\Phi:U\to \cM_{H}^s(r,L,L^2)\qquad W\mapsto E_W$$ where $U$ is a dense open subset of $\Gr(r+1,H^0(L))$.
\noindent 

Note that, by Kodaira vanishing, we have
$h^0(L) = \chi(L) =  \chi(\cO_S)  + \frac{1}{2}L^2,$
so we obtain
\begin{equation}
	\label{EQ:DimGr for K=0num}
	\dim(\Gr(r+1,H^0(L)))= (r+1)\left(\frac{1}{2} L^2+\chi(\cO_S)-r-1 \right).
\end{equation}

% Assume now that $S$ is not hyperelliptic  with $K_S$ of order $3$. We claim that, under this assumption, $\cM^s_H(r,L,L^2)$ is not empty and it is smooth of the expected dimension. In order to prove this, as recalled in Section \ref{SEC:preliminary results}, it is enough to show that $\Ext^2(E,E)_0=0$ for all $[E]\in \cM^s_H(r,L,L^2)$. For such $[E]$, recall that
% \begin{align*}
% 	\Ext^2(E,E)_0  &= \ker \left(h^2(tr) \colon \Ext^2(E,E) \to H^2(\cO_S)\right)\simeq \\
% 	&\simeq \ker \left(h^2(tr) \colon \Hom(E,E\otimes K_S)^* \to H^0(K_S)^*\right),
% \end{align*}
% where the latter isomorphism follows from Serre duality.
% \smallskip

% \begin{itemize}
%     \item If $K_S$ is trivial, we have $\Ext^2(E,E)_0\simeq \Hom(E,E)_0^*=0$ as $E$ is simple (see also \cite[page 168]{HL}. 
%     \item If $K_S$ is not trivial, then $H^2(\cO_S)\simeq H^0(K_S)^*=0$ by Serre duality. Then, $\Ext^2(E,E)_0\simeq \Ext^2(E,E)\simeq \Hom(E,E\otimes K_S)^*$. Now, both $E$ and $E\otimes K_S$ are $H$-stable vector bundles, with the same rank and the same $H$-slope, we have that $\Hom(E,E\otimes K_S)$ is trivial unless $E\simeq (E\otimes K_S)$. If this happens, then $c_1(E)=c_1(E)+rK_S$ so, $K_S$ is a torsion line bundle whose order divides $r$. It is well known that the possible orders for $K_S$ when $K_S\neq \cO_S$ are $2,3,4$ and $6$, with $3$ occurring only when $S$ is hyperelliptic (see \cite[page 188]{BPV}). This case has been excluded by assumption, so $r$ needs to be even. On the other hand, this cannot happen by Theorem \ref{THM:MainTHM2Ks for KodDim0}. 
% \end{itemize}

We recall (see  Section \ref{SEC:preliminary results}) that the moduli space  $\cM^s_H(r,L,L^2)$ is not empty and it is smooth of the expected dimension if $\Ext^2(E,E)_0=0$ for all $[E]\in \cM^s_H(r,L,L^2)$, where  
\begin{align*}
	\Ext^2(E,E)_0  &= \ker \left(h^2(tr) \colon \Ext^2(E,E) \to H^2(\cO_S)\right)\simeq \\
	&\simeq \ker \left(h^2(tr) \colon \Hom(E,E\otimes K_S)^* \to H^0(K_S)^*\right),
\end{align*}
and the latter isomorphism follows from Serre duality. We end up in one of the two possible cases:
\smallskip

\begin{itemize}
    \item If $K_S$ is trivial, we have $\Ext^2(E,E)_0\simeq \Hom(E,E)_0^*=0$ as $E$ is simple (see also \cite[page 168]{HL}). 
    \item If $K_S$ is not trivial, then $H^2(\cO_S)\simeq H^0(K_S)^*=0$ by Serre duality. Then, $\Ext^2(E,E)_0\simeq \Ext^2(E,E)\simeq \Hom(E,E\otimes K_S)^*$. Now, both $E$ and $E\otimes K_S$ are $H$-stable vector bundles, with the same rank and the same $H$-slope: we have that $\Hom(E,E\otimes K_S)$ is trivial unless $E\simeq (E\otimes K_S)$. If this happens, then $c_1(E)=c_1(E)+rK_S$ so, $K_S$ is a torsion line bundle whose order divides $r$. It is well known that the possible orders for $K_S$ when $K_S\neq \cO_S$ are $2,3,4$ and $6$, with the last three cases occurring only when $S$ is hyperelliptic (see \cite[page 188]{BPV}).
\end{itemize}

Following the above reasoning, we can conclude that if $S$ is a $K3$ or abelian surface then we have $\Ext^2(E,E)_0=0$, for any $[E] \in \cM^s_H(r,L,L^2)$. So the moduli space is smooth and its dimension, given by Equation \eqref{EQ:Edim}, is the following:
\begin{equation}
	\label{EQ:EDimM for K=0num}
	\dim(\cM^s_H(r,L,L^2))= (r+1)(L^2-(r-1)\chi(\cO_S)).
\end{equation}

Notice, in particular, that 
$$\dim(\Gr(r+1,H^0(L)))\leq \frac{1}{2}\dim(\cM^s_H(r,L,L^2))$$
with equality if and only if $\chi(\cO_S)=2$, i.e. if and only if $S$ is a $K3$ surface.
\begin{theorem}
	\label{THm:lagrangian}
	Let $S$ be a $K3$ surface and let $H$ be an ample primitive line bundle. If we choose $(r,m)$ and $L$ as in Theorem \ref{THM:MainTHM2Ks for KodDim0} then, whenever $\cM_H(r,L,L^2)$ is a smooth irreducible symplectic variety, 
	the closure of $\Ima(\Phi)$ in $\cM_H(r,L,L^2)$ is a (possibly singular) Lagrangian subvariety.
\end{theorem}
% \begin{theorem}
	% \label{THm:lagrangian}
	% Let $S$ be a $K3$ surface with $\rho(S)=1$ and let $H$ be an ample primitive line bundle with $H^2=4h$, $h\geq 1$. If we choose $(r,m)$ and $L$ as in Theorem \ref{THM:MainTHM2Ks for SGT} then, 
	% \begin{enumerate}[(a)]
		%     \item $\cM^s_H(r,L,L^2) = \cM_H(r,L,L^2)$, which is a smooth irreducible symplectic variety;
		%     \item the closure of $\Ima(\Phi)$ in $\cM_H(r,L,L^2)$ is a (possibly singular) Lagrangian subvariety.
		% \end{enumerate}
	% \end{theorem}

\begin{proof}
	Consider the Mukai vector $v=(r,L,r-L^2/2)=(r, mH, r-m^2\frac{H^2}{2})$ associated to our construction (see \cite{HL}). Then, with the notation in \cite{PR23}, we have $\cM_H(r,L,L^2)=\cM_v(S,H)$ and $\cM_H^s(r,L,L^2)=\cM_v^s(S,H)$.
	It is enough to recall that $Y=\overline{\Ima(\Phi)}$ is birational to $\Gr(r+1,H^0(L))$ and thus it is rational. Hence, (non-zero) holomorphic $2$-forms on $\cM_v(S,H)$ restrict to a trivial $2$-form on $Y_{reg}$. On the other hand, when $S$ is a $K3$ one has $\chi(\cO_S)=2$ and  we have
$$\dim(\cM_v(S,H))=\dim(\cM^s_H(r,L,L^2))=2\dim(\Gr(r+1,H^0(L)))=2\dim(Y)$$ so $Y$ is a Lagrangian subvariety of $\cM_v(S,H)$. 
\end{proof}

\begin{remark}
    \label{REM:Lagrangian}
	If $S$ has Picard number one, then $ \cM_H(r,L,L^2)$ is a smooth irreducible symplectic variety and Theorem \ref{THm:lagrangian}  applies whenever $(r,m)$ are coprime\footnote{This actually happens for all the pairs satisfying condition $A_3(1)$, i.e. the ones given in Lemma \ref{LEM:CondA31_KSnumtrivial}. There are also pairs that satisfy condition $A_3(2)$ for which $(r,m)$ are coprime.}.  
	\smallskip
	Indeed, since $H$ is primitive, we have that the Mukai vector $v=\left(r,mH, r-m^2 \frac{H^2}{2}\right)$ is primitive too.
	% We prove that $v$ is a primitive Mukai vector. Since $H$ is primitive, it is sufficient to prove that $r$ and $m$ are coprime numbers. Assume that $h=1$, so that $r=4a+1$ and $m=3a+1$. Then 
	%     \[
	%     -3r+4m=-12a-3+12a+4=1. 
	%     \]
	%     Instead, let us suppose that $h\geq 2$, so that $r=4ah-2h-1$ and $m=(1+2h)a-h-1$. Then 
	%     \[
	%       -(1+2h)r+4hm=-(1+2h)(4ah-2h-1)+4h((1+2h)a-h-1)=1. 
	%     \]
	% In both cases, $r$ and $m$ are always coprime.    
	Since $\rho(S)=1$, then $H$ is both $v$-generic and general with respect to $v$ (see \cite[Lemma 2.9]{PR23}). Then, by \cite[Theorem 1.10]{PR23} (and see also \cite[Theorem 4.4]{KLS2006}), $\cM_v(S,H)$ is an irreducible simplectic variety. The smoothness follows from the fact that $v$ is primitive and $H$ is general with respect to $v$ (see either the remark following Theorem 1.10 in \cite{PR23} or \cite[Lemma 2]{Saw16}).
\end{remark}

A different explicit description of the Lagrangian subvariety of the moduli space  of stable vector bundles on a smooth regular algebraic surface with  $p_g >0$  can be found in  \cite{Y93}.

% -------------------------------------------------------------
% -------------------------------------------------------------

\subsection{Del Pezzo surfaces}
Let us assume that $S$ is a del Pezzo surface and let $e$ be its degree (i.e. $e=K_S^2$). We recall (see \cite{Dem77}, for example) that, although $-K_S$ is ample, it is not very ample when $e\leq 2$. On the other hand, $-2K_S$ is always globally generated, and it is very ample unless $e=1$, whereas $-3K_S$ is always very ample.
We set
$$L= -mK_S  \quad H = -3K_S$$
with $m\geq 2$ so that $L$ is nef, big and globally generated, and there exists a smooth irreducible curve $C$ in the linear system $\vert H \vert$ (i.e. assumption $A_1$ holds). Moreover, as $L-H=(m-3)K_S$, it follows that $H^1(L-H)=0$ (by Kodaira vanishing for $m\neq 3$ and since $S$ is regular, for the case $m=3$), hence the restriction map of global section $\rho \colon H^0(L) \to H^0(L|_C)$ is surjective (so that assumption $A_2$ holds). 
\smallskip

We would like to find pairs $(r,m)$ such that assumption $A_3$ holds too for some integer $r \geq 2$. For this class of surfaces, for brevity, we focus only on assumption $A_3(1)$. Analogous results hold for the case $A_3(2)$ and can be easily obtained.
\smallskip

As $g(C)=1+3K_S^2=1+3e\geq 4,$ the condition $\deg(L|_C)=rg+1$ is equivalent to $3me=r(1+3e)+1$. Then, one has necessarily 
\begin{equation}
	\label{EQ:Valuesof r and m for delPezzo}
	r=3ae-1\quad \mbox{ and }\quad m=a(3e+1)-1, \qquad \mbox{ with } a\geq 1.
\end{equation}
Notice that, under our assumptions, we have no solution when $m=2$. So we have the following result: 

\begin{theorem}
	\label{THM:MainTHM for FanoSurf}
	Let $S$ be del Pezzo surface of degree $e$. For any natural number $a\geq 1$ we consider a pair $(r,m)$ as in \eqref{EQ:Valuesof r and m for delPezzo}. Then, if $L=-mK_S$ and $H=-3K_S$, to the triple $(S,L,H,r)$ is admissible and there exists a subvariety of $\cM_{H}^s(r,L,L^2)$ birational to $\Gr(r+1,H^0(L))$. 
\end{theorem}

% -------------------------------------------------------------
% -------------------------------------------------------------

\subsection{Elliptic surfaces}
Here, we present two classes of examples of admissible data for elliptic surfaces, focusing specifically on assumption $A_3(1)$. These surfaces are of \textit{product-quotient} type, a class that has been extensively investigated in the literature (see, e.g., the recent works \cite{Fede1}, \cite{Fede2}, and \cite{AleFraGle}).
\smallskip

First of all, let us consider a surface $S=E\times F$ where $E$ is an elliptic curve and $F$ is a curve of genus $g\geq 2$ so that $S$ is an elliptic surface. We write $K_F$ to mean any canonical divisor on $F$. Hence, $K_F$ is globally generated and ample and the same holds for the divisor $2p$ on $E$, where $p$ is any point on $E$. Then, if we set $H=2(p\times F)+E\times K_F$, we have that $H$ is globally generated and ample so that assumption $A_1$ holds by Bertini's Theorem. It is easy to see that any smooth curve $C$ in $|H|$ has genus $g(C)=6g-5$.
\smallskip

If we set $L=mH$ with $m\geq 3$ we have $L-H\equiv K_S+D$ with $D$ ample, so $H^1(L-H)=0$ by Kodaira vanishing: assumption $A_2$ holds.

Reasoning as in the other cases, after some computation, one proves the following:

\begin{theorem}
	\label{THM:MainTHM2Ks for EllFib}
	If $S=E\times F$ with $E$ an elliptic curve and $g=g(F)\geq 2$, set $H=2(p\times F)+E\times K_F$. For any integer $a\geq (7g-3g^2)/(6g-5)$ we consider a pair $(r,m)$ such that 
	\[
	\left.\begin{aligned}
		r &= 4g^2-10g+5+8(g-1)a      &\qquad \mbox{and}\qquad& m = 3g^2-7g+3+a(6g-5).
	\end{aligned}\right.
	\]
	Then, if we set $L= mH$, then $(S,L,H,r)$ is admissible and there exists a subvariety of $\cM_{H}^s(r,L,L^2)$ birational to $\Gr(r+1,H^0(L))$. 
\end{theorem}

For the second class of examples, we slightly modify the previous one. Let us consider a finite group $G$ acting faithfully both on an elliptic curve $E$ and a smooth curve $F$ of genus $g\geq 2$, such that $E/G\cong \bP^1$ and $F/G$ is an elliptic curve. Assume that the action of the diagonal subgroup $\Delta\leq G\times G$ on the product $E\times F$ is free, so that the quotient $S:=(E\times F)/\Delta$ is smooth. Surfaces of this type are said to be \textit{isogenous to a product of curves} (see \cite[Def. 3.1]{Cat00}). We have the following hexagonal commutative diagram 
\[\begin{tikzcd}
	& {E\times F} \\
	E & S & F \\
	{\mathbb P^1} && {F/G} \\
	& {\mathbb P^1\times F/G}
	\arrow["{p_1}"', from=1-2, to=2-1]
	\arrow["{\lambda_{12}}", from=1-2, to=2-2]
	\arrow["{p_2}", from=1-2, to=2-3]
	\arrow["{\lambda_1}"', from=2-1, to=3-1]
	\arrow["{f_1}"', from=2-2, to=3-1]
	\arrow["{f_2}", from=2-2, to=3-3]
	\arrow["\lambda", from=2-2, to=4-2]
	\arrow["{\lambda_2}", from=2-3, to=3-3]
	\arrow[bend left=25, "{\eta_1}", from=4-2, to=3-1]
	\arrow[bend right=25, "{\eta_2}"', from=4-2, to=3-3]
\end{tikzcd}\]
involving projections from products (namely $p_1,p_2, \eta_1$ and $\eta_2$), quotients by the various actions of $G$ (namely $\lambda_1,\lambda_2$ and $\lambda_{12}$) and the natural fibrations $f_1$ and $f_2$ induced on $S$. %{\blue PER FEDE: $\lambda$ non è un quoziente in generale, giusto? era solo per descrivere alla buona il diagramma. Fai un check se non abbiamo scritto vaccate.}
%{\red Risposta: No la mappa $\lambda$ non è in generale un quoziente per l'azione di un gruppo, ma solamnete una mappa finita di grado $\vert G\times G/\Delta\vert =\vert G\vert$. Il morfismo diventa di Galois se e solo se la diagonale $\Delta$ è normale in $G\times G$, ovvero se e solo se $G$ è abeliano. Ad ogni modo, $\lambda$ è sempre ben definita e il diagramma di sopra esiste sempre.}
%{\color{blue} FIL: Perfetto, mi torna tutto. Mi sembra che le cose, così come sono, vadano bene. Togliamo la parte colorata e lo dichiariamo a posto così, ok?}
%{\red vai togli pure :D}
\smallskip

Consider $p\in \bP^1$ and $q\in F/G$ and the fibers $F_1:=f_1^*(p)$ and $F_2:=f_2^*(q)$. Clearly, the choice of the two points doesn't matter if we are only interested in the numerical class of $F_1$ and $F_2$. We observe that $S$ has Kodaira dimension one as $\lambda_{12}\colon E\times F\to S$ is a finite étale morphism of smooth surfaces. The numerical class of a canonical divisor of $S$ is 
\[K_S\equiv_{num} \frac{2g-2}{\vert G\vert}F_2.
\]
We observe that any irreducible curve $C$ of $S$ such that $C\cdot F_1=0$ is contained in a fiber of $f_1$; otherwise, we could always pick up a point of $C$ such that the fiber of that point and $C$ intersect positively. A similar argument holds when $C\cdot F_2=0$. 

Let us consider a divisor $H:=F_1+2F_2$. Since $F_1$ and $F_2$ are nef divisors (as $f_1$ and $f_2$ are fibrations), then $H$ is ample by the previous argument. Indeed,  $C\cdot H\geq 0$ with equality if and only if both $C\cdot F_1$ and $C\cdot F_2$ are zero, a contradiction. The divisor $H$ is also globally generated as $F_1$ and $2F_2$ are globally generated as pullback, via a dominant morphism $\lambda$, of the globally generated divisor $\{p\}\times (F/G)+2\bP^1\times \{q\}$. In particular, there exists a smooth curve $C\in |H|$ and $A_1$ holds. It is easy to see that the genus of any smooth curve $C\in \vert H\vert$ is $g(C)=2\vert G\vert+g$. 

If we set $L:=mH$ with $m\geq 1+(g-1)/|G|$, then $L-H-K_S$ is ample, so that $H^1(L-H)$ is zero by Kodaira vanishing and $A_2$ hold.  
\begin{theorem}
	\label{THM:MainTHM2Ks for EllFib2}
	Let $S=(E\times F)/G$ isogenous to a product of a an elliptic curve $E$ and a curve $F$ with $g=g(F)\geq 2$. Set $H=F_1+2F_2$ as above. For any integer $a$, consider the pair $(r,m)$ with 
	\[
	\left.\begin{aligned}
		r &= \frac{2(2a+1)|G|-1}{g}      &\qquad \mbox{and}\qquad& m = \frac{(2a+1)(2|G|+g)-1}{2g}.
	\end{aligned}\right.
	\]
	Then, if $r$ and $m$ are integers with $r\geq 2$ and $m \geq 2$, and if we set $L= mH$, then $(S,L,H,r)$ is admissible so there exists a subvariety of $\cM_{H}^s(r,L,L^2)$ birational to $\Gr(r+1,H^0(L))$. 
\end{theorem}

\end{document}